\documentclass{amsart}

\usepackage{amssymb,amscd,amsthm,amsxtra}
\usepackage{dsfont,latexsym}

\newtheorem{thm}{Theorem}

\newtheorem{lemma}[thm]{Lemma}

\theoremstyle{definition}

\theoremstyle{remark}

\newcommand{\R}{\mathds{R}}

\newcommand{\Z}{\mathds{Z}}

\newcommand{\CC}{\mathcal{C}}

\newcommand{\EE}{\mathcal{E}}
\newcommand{\KK}{\mathcal{K}}

\usepackage{mathrsfs}
\renewcommand{\mathcal}{\mathscr}

\renewcommand{\le}         {\leqslant}

\renewcommand{\ge}         {\geqslant}

\begin{document}

\author{Ovidiu Savin and Enrico Valdinoci}
\thanks{OS has been supported by 
NSG grant 0701037. EV has been
supported
by
FIRB 
project ``Analysis and Beyond''
and GNAMPA
project
``Equa\-zio\-ni non\-li\-nea\-ri su va\-rie\-t\`a:
pro\-prie\-t\`a qua\-li\-ta\-tive e clas\-si\-fi\-ca\-zio\-ne
del\-le so\-lu\-zio\-ni''.
Part of this work
was carried out
while EV was visiting Columbia
University. We thank Guido de Philippis for an interesting
discussion.}

\title[Density estimates
via the Sobolev inequality
]{Density estimates \\ for a nonlocal variational model \\
via the Sobolev inequality}

\begin{abstract}
We consider the minimizers of the energy 
$$ \|u\|_{H^s(\Omega)}^2+\int_\Omega W(u)\,dx,$$ with $s \in 
(0,1/2)$, 
where $\|u\|_{H^s(\Omega)}$ denotes the total contribution from $\Omega$ 
in the $H^s$ norm of $u$, and $W$ is a double-well potential.
By using a fractional Sobolev inequality,
we give a new proof of the fact 
that the sublevel sets 
of a minimizer~$u$ in a large ball~$B_R$ occupy
a volume comparable with the volume of~$B_R$.
\end{abstract}

\maketitle

Given~$s\in(0,1/2)$ and~$\Omega\subseteq\R^n$, with~$n\ge2$,
we define
$$ \KK(u;\Omega):=\frac{1}{2}\int_\Omega\int_\Omega\frac{
|u(x)-u(y)|^2
}{
|x-y|^{n+2s}}\,dx\,dy
+\int_\Omega\int_{\CC\Omega}\frac{
|u(x)-u(y)|^2
}{
|x-y|^{n+2s}}\,dx\,dy.$$
We take~$W$ to be a double-well potential, more
precisely, we assume that $W:[-1,1] \to 
[0, \infty)$,
\begin{equation}\label{Wcond}
W \in C^2([-1,1]), \quad W(\pm 1)=0, \quad W>0 \quad \mbox{in $(-1,1)$}
\end{equation}
$$W'(\pm 1)=0, \quad  {\mbox{and}}\quad
W''(\pm 1)>0.$$
The energy functional we are interested in is the sum
of the nonlocal contribution given by~$\KK$ and a local one
induced by~$W$, i.e., we define
\begin{equation}\label{EE}
\EE(u; \Omega):=\KK(u; \Omega) + \int_{\Omega} W(u(x))
\, dx.\end{equation}
We say that $u$ is a minimizer
in $\Omega$ if~$\EE (u;\Omega)<\infty$ and
$$ \EE (u;\Omega) \le \EE (v;\Omega)$$ for any $v$ which coincides with 
$u$ in $\CC \Omega$.
It is easy to see that minimizers satisfy an Euler-Lagrange
equation of nonlocal type, which is driven by an elliptic integral
operator of fractional type. More precisely, a minimizer~$u$
is a solution of
$$ (-\Delta)^s u(x)+W'(u(x))=0,$$
where~$(-\Delta)^s$ is the fractional power
of the positive operator~$-\Delta$, up to a normalizing multiplicative
constant, that we neglect.
More explicitly,
$$(-\Delta)^s u(x):=\int_{\R^n}\frac{u(x)-u(y)}{|x-y|^{n+2s}}dy$$
and the integral is understood in the principal value sense.

{F}rom the physical point of view, the functional in~\eqref{EE}
may be seen as a nonlocal extension of the classical Allen-Cahn
model for phase coexistence (for the latter,
see, e.g., \cite{Gur}), and 
its interfaces may be
related with suitable nonlocal minimal surfaces.
Roughly speaking, the double-well
potential~$W$ tries to drive the minimizers
towards the pure phases~$-1$ and~$+1$; on the other
hand, up to scaling the
space variables, the term~$\KK$ may be seen as
a penalization which prevents the formation of
unnecessary interfaces and makes the problem consistent
from the mathematical point of view. The main difference
between~\eqref{EE} and the classical Allen-Cahn model is that
we have here a fully nonlocal interaction~$\KK$ instead of
the usual ``kinetic term'' $\int_\Omega |\nabla u(x)|^2\,dx$.

{F}rom the mathematical point of view, the term~$\KK$
may also be regarded as the square of a (semi)norm in
a fractional Sobolev space~$H^s$ (say,~$\|u\|_{H^s(\Omega)}^2$,
see, e.g.,~\cite{BrMi}
and the more comprehensive bibliography on this quoted there).

We refer to~\cite{ADP, CRS, CV} for
the precise definitions and some basic results on
the nonlocal minimal surfaces
linked with the limit interface of
the functional in~\eqref{EE}, and to~\cite{SaV, SaV2}
for a more detailed discussion and motivation.
For related nonlocal problem of phase segregation
with physical importance, see
also~\cite{ABS, Bou, GP, Gon}.
Moreover, we recall that
the nonlocal contribution~$\KK$
may be seen as arising from a long range
interaction of particles, in connection
with some statistical mechanics model
(see, e.g.,~\cite{Orl}).

In~\cite{SaV2}, we have obtained, among others, the following
result:

\begin{thm}\label{DE}
Let~$u$ be a minimizer of $\EE$ in~$B_R$. Then for any~$\theta_1,$ 
$\theta_2\in(-1,1)$ such that
\begin{equation}\label{MUCO1}
u(0)>\theta_1, \end{equation}
we have that
\begin{equation}\label{hard}
\big| \{ u>\theta_2\}\cap B_R\big|\,\ge\, \overline{c} \,R^n
\end{equation}
if~$R\ge \overline{R}(\theta_1,\theta_2)$.

The constant~$\overline{c}>0$ 
depends only on $n$, $s$ and $W$ and~$\overline{R}(\theta_1,\theta_2)$ is 
a large constant that depends also on $\theta_1$ and $\theta_2$.
\end{thm}

Theorem~\ref{DE} follows in the streamline of the
density estimates for sublevel sets of minimizers, whose
study was started in~\cite{CC} for the Allen-Cahn equation.
Estimate~\eqref{hard} is quite meaningful in applications,
since, from the physical viewpoint, it represents
an estimate on the probability of finding a phase in
a given portion of the medium.
We refer 
to~\cite{SaV2}
for further references about related density estimates,
and for further comments about the important
consequence that these estimates have in the applications
(such as geometric results on $\Gamma$-convergence,
uniform convergence of rescaled interfaces, rigidity and
symmetry properties, etc.). 

Theorem~\ref{DE} has been proved in~\cite{SaV2}
by using a fine estimate on a weighted double integral.
The purpose of this note is to perform an alternative
proof, by using a fractional Sobolev inequality.

Such an alternative proof is given below.
Then, in the appendix, we give
a simple, essentially self-contained, proof of
the fractional Sobolev inequality in use.

The preliminary computations for the proof
of Theorem~\ref{DE}
are in common with~\cite{SaV2}, but the rest of
the proof given here
is conceptually and technically
quite different. Indeed, the proof given in~\cite{SaV2}
is more general (for instance, it works
also for the case~$s\in[1/2,1)$, in which
the limit interface of the nonlocal problem is
the classical, i.e. local, perimeter, and
the technique used also plays an important role
in the study of the $\Gamma$-convergence
of the rescaled
functional performed in~\cite{SaV}).
On the other hand,
the proof in~\cite{SaV2} is somewhat more
difficult, since it is based on an ``ad hoc''
fine measure theoretic result, namely
Theorem~1.6 there, which, roughly
speaking, estimates the energy
for making the phase change based on the
nonlocal integral contribution.
In this paper, this will be achieved more directly,
via a fractional Sobolev inequality, and
this makes the argument technically
simpler (and also
closer in spirit to the proof of~\cite{CC}
for the classical Allen-Cahn model, where
the standard Sobolev inequality was used).
\bigskip

\begin{proof}[Proof of Theorem~\ref{DE}]
First of all, we observe that, by~\eqref{Wcond},
there exists a small constant~$c>0$ such that\footnote{It may
be worth to remark that, in fact, Theorem~\ref{DE} is proven
here simply under condition~\eqref{grow}, which is weaker 
than~\eqref{Wcond}.}
\begin{equation}\label{grow}
\begin{split}
&{\mbox{$W(t)\ge W(r)+c(1+r)(t-r)+c(t-r)^2$ when
$-1\le r\le
t\le
-1+c$}}\\
&{\mbox{and $W(r)-W(t)\le (1+r)/c$ when $-1\le r\le t\le+1$.}}
\end{split}
\end{equation}
We fix~$\theta_\star:= 
\min\{\theta_1,\,\theta_2,\, -1+c\}$,
with~$c$ as in~\eqref{grow}.

Now, we recall a useful barrier
that was constructed
in Lemma~3.1 in~\cite{SaV2}:

\begin{lemma}\label{tau}
Given any~$\tau>0$, there exists
a constant $C>1$, possibly
depending on~$n$, $s$ and $\tau$,
such that the following holds: for any $R\ge
C$, there exists a rotationally symmetric function
\begin{equation*}
w\in C\big( \R^n ,[-1+CR^{-2s},\,1]\big),
\end{equation*}
with
\begin{equation*}{\mbox{
$w=1$ in~$\CC B_R$,}}\end{equation*}
such that
\begin{equation}\label{al1} -(-\Delta)^s u(x)=
\int_{\R^n}\frac{w(y)-w(x)}{{|x-y|^{n+2s}}}\,dy
\le \tau\big(1+w(x)\big)\end{equation}
and
\begin{equation}\label{al2}
\frac1C \big( R+1-|x|\big)^{-2s}
\le 1+w(x)\le C \big( R+1-|x|\big)^{-2s}
\end{equation}
for any~$x\in B_R$.
\end{lemma}

As in~\cite{SaV2}, we define
$$ A(R)\,:=\,c\,\int_{ B_R \cap \{w<u\le\theta_\star\}}
(u-w)^2\,dx \qquad{\mbox{
and
}}\qquad
V(R):=\big| \{ u>\theta_\star\}\cap
B_R\big|.$$
Also,
we fix~$K\ge 2(R_o+1)$, to be taken suitably
large in the sequel (possibly
in dependence of~$\theta_\star$), where~$R_o$ is also 
suitably large (at the end, roughly speaking,
this will give~$\overline R$
as in
the statement of Theorem~\ref{DE},
up to constants), and~$R>2K$.
We take~$w$ to be the function constructed in Lemma~\ref{tau}
with~$\tau:=c/4$, and
we define~$v(x):=\min\{u(x),w(x)\}$.

Then, from formula~(3.34) of~\cite{SaV2}, one knows that
\begin{equation}\label{la6}
\begin{split}&
\!\! A(R)+\KK(u-v;B_R)
+\frac{c}2\int_{B_R\cap \{w<u\le \theta_\star\}} (1+w)(u-w)\,dx
\\
\\ &\le C \int_{0}^R (R+1-t)^{-2s} V'(t)\,dt.
\end{split}
\end{equation}
Now, we observe that~$u-v=0$ outside~$B_R$, and
so the Sobolev-type inequality (see Theorem~\ref{SOBOLEV})
gives that
\begin{equation}\label{la7}\begin{split}&
\KK(u-v;B_R)=\KK(u-v,\R^n)\\ &\qquad
\ge c_1 \|u-v\|_{L^{2n/(n-2s)} (\R^n)}^2=
c_1 \|u-v\|_{L^{2n/(n-2s)} (B_R)}^2
\end{split}\end{equation}
for a suitable~$c_1>0$.

Now, by~\eqref{al2},
\begin{equation}\label{RON}{\mbox{
$w<-1+(1+\theta_\star)/2$ in~$B_{R-K}$,}}\end{equation}
as long as~$K$ is large enough. Hence, we have that
$$ |u-v|\ge u-v\ge u-w\ge (1+\theta_\star)/2\qquad{\mbox{
in $B_{R-K}\cap
\{ u>\theta_\star\}$}}$$
and so
\begin{eqnarray*}&&
\|u-v\|_{L^{2n/(n-2s)} (B_R)}^2 \ge\left(\int_{B_{R-K}\cap 
\{ u>\theta_\star\}}
|u-v|^{2n/(n-2s)}
\right)^{(n-2s)/n}
\\ &&\qquad\ge
\left(\int_{B_{R-K}\cap \{ u>\theta_\star\}}
((1+\theta_\star)/2)^{2n/(n-2s)}
\right)^{(n-2s)/n}
=c_2 V(R-K)^{(n-2s)/n}
,\end{eqnarray*}
for a suitable~$c_2>0$ (possibly depending on~$\theta_\star$,
hence on~$\theta_1$ and~$\theta_2$, which have been
fixed at the beginning).

Then, recalling~\eqref{la6}
and~\eqref{la7}, we obtain that
\begin{eqnarray}\label{la8}
{c_3} V(R-K)^{(n-2s)/n}\le
\int_{0}^R (R+1-t)^{-2s} V'(t)\,dt,
\end{eqnarray}
for a suitable~$c_3>0$. Now,
we integrate~\eqref{la8} in~$R\in [\rho,(3/2)\rho]$, with~$\rho\ge 2K$,
and we use that~$s\in(0,1/2)$
to obtain that
\begin{eqnarray*}&&
\frac{c_3}2\, \rho \, V(\rho-K)^{(n-2s)/n} \\&\le& c_3
\int_\rho^{(3/2)\rho} V(R-K)^{(n-2s)/n}\,dR
\\ &\le&\int_\rho^{(3/2)\rho}\left(
\int_{0}^R (R+1-t)^{-2s} V'(t)\,dt\right)\,dR\\ &\le&
\int_0^{(3/2)\rho}\left(\int_t^{(3/2)\rho}
(R+1-t)^{-2s} 
\,dR\right) V'(t)
\,dt\\
&=& \frac{1}{1-2s}
\int_0^{(3/2)\rho}\Big[
\big( (3/2)\rho+1-t\big)^{1-2s}
-
1\Big] V'(t)
\,dt \\ &\le&
\frac{ \big((3/2)\rho+1\big)^{1-2s}}{1-2s}
\int_0^{(3/2)\rho} V'(t)
\,dt\\ &\le& \frac{4^{1-2s}}{1-2s} 
\,\rho^{1-2s}\,V((3/2)\rho),
\end{eqnarray*}
that is, for any~$\rho\ge 2K$,
\begin{equation}\label{o-it}
\rho^{2s} V(\rho-K)^{(n-2s)/n}\le \tilde C 
V((3/2)\rho),\end{equation}
for a suitable~$\tilde C>0$.

We take~$r:=\rho-K$ in~\eqref{o-it},
and we obtain that
\begin{equation}\label{o-it-2}
r^{2s} V(r)^{(n-2s)/n}\le C
V(2r),\end{equation}
for a suitable~$C\ge K^2$,
as long as~$r\ge C$ (notice that~$C$
may depend on~$K$, which is now fixed
once and for all).

Now, we recall the following
general, inductive result, for the proof
of which we refer to Lemma~3.2 of~\cite{SaV2}:

\begin{lemma}\label{lemma indu}
Let~$\sigma$, $\mu\in(0,+\infty)$, $\nu\in(\sigma,+\infty)$
and~$\gamma$, $R_o$, 
$C\in(1,+\infty)$.

Let~$V:(0,+\infty)\rightarrow(0,+\infty)$ be a
nondecreasing function. For any~$r\in [R_o,+\infty)$, let
$$ \alpha(r):=\min\left\{1,\,\frac{\log V(r)}{\log r}
\right\}.$$
Suppose that
\begin{equation*}
V(R_o)\ge\mu
\end{equation*}
and
\begin{equation*}
r^\sigma\,\alpha(r)\,V(r)^{(\nu-\sigma)/\nu}\le C V(\gamma r),\qquad
{\mbox{ for any $r\in[R_o,+\infty)$.}}
\end{equation*}
Then, there exist~$c\in(0,1)$ and~$R_\star\in[R_o,+\infty)$, possibly 
depending
on~$\mu$, $\nu$, $\gamma$, $R_o$ and~$C$, such that
$$ V(r)\ge cr^\nu,\qquad
{\mbox{ for any $r\in[R_\star,+\infty)$.}}$$
\end{lemma}

Then, we apply Lemma~\ref{lemma indu}
with~$\sigma:=2s$ and~$\gamma:=2$ and we deduce from~\eqref{o-it-2}
that~$V(R)\ge c_o R^n$ for large~$R$, for a suitable~$c_o\in(0,1)$.

Therefore, if we define~$\theta^\star:=\max\{\theta_1,\,\theta_2,\, 
-1+c\}$,
\begin{equation}\label{PP1}\begin{split}
& \big| \{ u>\theta^\star\}\cap B_R\big|+
\big| \{ \theta_\star<u\le \theta^\star
\}\cap B_R\big|\\
&\qquad=\big| \{ u>\theta_\star\}\cap B_R\big| = V(R)\ge c_o 
R^n\end{split}
\end{equation}
for large~$R$.
On the other hand, by Theorem~1.3 of~\cite{SaV2}, we
know that
$$ \EE(u;B_R) \le \overline{C} \,R^{n-2s},$$
for some~$\overline{C}>0$, and so
\begin{equation}\label{PP2}\begin{split}
&\overline{C} R^{n-2s}\ge \EE(u,B_R)\ge
\int_{ \{ \theta_\star<u\le \theta^\star
\}\cap B_R } W(u(x))\,dx
\\&\qquad \ge\,\inf_{r\in[\theta_\star,
\theta^\star]} W(r)\;
\big| \{ \theta_\star<u\le \theta^\star
\}\cap B_R\big|
.\end{split}\end{equation}
By~\eqref{PP1} and~\eqref{PP2},
we obtain that~\eqref{hard} holds, thus proving Theorem~\ref{DE}.
\end{proof}

\section*{Appendix -- The fractional Sobolev inequality}

For completeness, we provide here an essentially selfcontained
and elementary proof of the
Sobolev-type inequality used in this paper
(in particular, we will not make use of
neither interpolations or Besov spaces).
For a more comprehensive treatment of fractional Sobolev-type
inequalities see~\cite{BrMi} and references therein.

In this appendix,
we will fix $s\in(0,1)$
(in fact, when~$s\in[1/2,1)$ some of the statements
may be strengthened, see~\cite{Br}).
We recall an elementary estimate,
for the proof of which see, e.g., the Appendix of~\cite{SaV2}
and, in particular, Lemma~A.1 there:

\begin{lemma}\label{5yhh}
Fix $x\in\R^n$. Let $E\subset\R^n$ be a measurable set with finite
measure.
Then,
$$ \int_{\CC E}\frac{dy}{|x-y|^{n+2s}}\ge c(n,s)\,|E|^{-2s/n},$$
for a suitable constant $c(n,s)>0$.
\end{lemma}

Now, we make a general observation about a useful summability
property:

\begin{lemma} \label{OS1}
Fix~$T> 1$. Let~$N\in\Z$ and
\begin{equation}\label{ak}\begin{split}\\ &{\mbox{
$a_k$
be a bounded, nonnegative, decreasing sequence}}\\&\qquad{\mbox{
with~$a_k=0$ for any $k\ge N$.}}
\end{split}\end{equation}
Then,
$$ \sum_{k\in\Z} a_k^{(n-2s)/n} T^{k}\le C(n,s,T)\,
\sum_{{k\in\Z}\atop{a_k\ne0}}
a_{k+1} a_k^{-2s/n} T^{k},
$$
for a suitable constant $C(n,s,T)>0$, independent
of~$N$.\end{lemma}

\begin{proof} 
By~\eqref{ak},
\begin{equation}\label{are c}
{\mbox{both }}
\sum_{k\in\Z} a_k^{(n-2s)/n} T^{k}
{\mbox{ and }}
\sum_{{k\in\Z}\atop{a_k\ne0}}
a_{k+1} a_k^{-2s/n} T^{k}
{\mbox{ are convergent series.}}
\end{equation}
Moreover, since $a_k$ is nonnegative and
decreasing, we have that
if $a_{k}=0$, then $a_{k+1}=0$. Accordingly,
$$ \sum_{k\in\Z} a_{k+1}^{(n-2s)/n} T^{k}=
\sum_{{k\in\Z}\atop{a_k\ne 0}} a_{k+1}^{(n-2s)/n} T^{k}.$$
Therefore, we use the H\"older
inequality with exponents $\alpha:=n/2s$ and $\beta:=n/(n-2s)$
as follows:
\begin{eqnarray*}
&& \frac{1}{T}
\sum_{k\in\Z} a_{k}^{(n-2s)/n} T^{k}
=
\sum_{k\in\Z} a_{k+1}^{(n-2s)/n} T^{k}\\&&\qquad=
\sum_{{k\in\Z}\atop{a_k\ne 0}} a_{k+1}^{(n-2s)/n} T^{k}
\\&&\qquad=
\sum_{{k\in\Z}\atop{a_k\ne0}} \Big( a_{k}^{2s/(n\beta)}
T^{k/\alpha}\Big) \Big( a_{k+1}^{1/\beta} a_{k}^{-2s/
(n\beta)} T^{k/\beta}\Big)\\
&&\qquad\le
\left(
\sum_{k\in\Z} \Big( a_{k}^{2s/(n\beta)}
T^{k/\alpha}\Big)^\alpha\right)^{1/\alpha} \left(
\sum_{{k\in\Z}\atop{a_k\ne 0}}\Big( a_{k+1}^{1/\beta}
a_{k}^{-2s/(n\beta)} T^{k/\beta}\Big)^\beta\right)^{1/\beta}
\\
&&\qquad\le
\left(
\sum_{k\in\Z} a_{k}^{(n-2s)/n} T^{k}\right)^{2s/n} \left(
\sum_{{k\in\Z}\atop{a_k\ne0}} a_{k+1}
a_{k}^{-2s/n} T^{k} \right)^{(n-2s)/n}.
\end{eqnarray*}
So, recalling \eqref{are c},
we obtain the desired
result.
\end{proof}

We use the above tools to
deal with the measure theoretic properties of the
level sets of the functions:

\begin{lemma}\label{OS2}
Let 
\begin{equation}\label{XX}{\mbox{
$f\in L^\infty(\R^n)$ be compactly
supported. }}\end{equation}
For any $k\in\Z$ let
$$ a_k:=
\big| \{ |f|>2^k\} \big|.$$
Then,
$$ \int_{\R^n}\int_{\R^n} \frac{|f(x)-f(y)|^2}{|x-y|^{n+2s}}\,dx\,dy\ge
c(n,s) \sum_{{k\in\Z}\atop{a_k\ne0}}
a_{k+1} a_k^{-2s/n} 2^{2k},$$
for a suitable constant $c(n,s)>0$.
\end{lemma}

\begin{proof} Notice that
$$ \big| |f(x)|-|f(y)|\big|\le |f(x)-f(y)|,$$
and so,
by possibly replacing~$f$ with~$|f|$, we may consider the case in
which~$f\ge0$.

We define
\begin{equation}\label{t6}
A_k:=\{ |f|>2^k\}.\end{equation}
We remark that~$A_{k+1}\subseteq A_k$, hence
\begin{equation}\label{X0}
a_{k+1}\le a_k .
\end{equation}
We define
$$ D_k:=A_k\setminus A_{k+1}=\{ 2^k<f\le 2^{k+1}\}
\qquad{\mbox{
and }}\qquad d_k:=|D_k|.$$
Notice that
\begin{equation}\label{YY}{\mbox{$d_k$
and $a_k$ are bounded and they become zero when
$k$ is large enough,}}\end{equation} thanks to~\eqref{XX}. Also,
we observe that the $D_k$'s are disjoint, that
\begin{equation}\label{X1}
\bigcup_{{\ell\in\Z}\atop{\ell\le k}} D_\ell\,=\,\CC A_{k+1}
\end{equation}
and that
\begin{equation}\label{X2}
\bigcup_{{\ell\in\Z}\atop{\ell\ge k}} D_\ell\,=\, A_{k}.
\end{equation}
As a consequence of \eqref{X2}, we have that
\begin{equation}\label{YYY}
a_k=\sum_{{\ell\in\Z}\atop{\ell\ge k}} d_\ell
\end{equation}
and so
\begin{equation}\label{X3}
d_k=a_k -\sum_{{\ell\in\Z}\atop{\ell\ge k+1}} d_\ell.
\end{equation}
We stress that the series in \eqref{YYY}
is convergent,
due to~\eqref{YY}, thus so is the series
in~\eqref{X3}. Similarly, we can define
the convergent series
\begin{equation}\label{SS4}
S:=\sum_{{\ell\in\Z}\atop{
a_{\ell-1}\ne 0
}} 2^{2\ell} a_{\ell-1}^{-2s/n}d_\ell.\end{equation}
We notice that~$D_k\subseteq A_k\subseteq A_{k-1}$,
hence~$a_{i-1}^{-2s/n}
d_\ell \le a_{i-1}^{-2s/n} a_{\ell-1}$. Therefore
\begin{equation}\label{FT}\begin{split}&
\Big\{ (i,\ell)\in \Z {\mbox{ s.t. }}
a_{i-1}\ne 0 {\mbox{ and }}
a_{i-1}^{-2s/n} d_\ell \ne 0\Big\} \\ &\qquad
\,\subseteq\,
\Big\{ (i,\ell)\in \Z {\mbox{ s.t. }}
a_{\ell-1}\ne 0 \Big\}.\end{split}\end{equation}
We use \eqref{FT}
and~\eqref{X0} in the following computation:
\begin{equation}\label{X5}\begin{split}
& \sum_{ {i\in\Z}\atop{a_{i-1}\ne0}}
\sum_{ {\ell\in\Z}\atop{\ell\ge i+1} } 2^{2i} a_{i-1}^{-2s/n} d_\ell
=
\sum_{{i\in\Z}\atop{a_{i-1}\ne0}}
\sum_{ {\ell\in\Z}\atop{ {\ell\ge i+1}\atop{
a_{i-1}^{-2s/n} d_\ell \ne 0} }} 2^{2i} a_{i-1}^{-2s/n}
d_\ell
\\ &\qquad\le
\sum_{i\in\Z}
\sum_{ {\ell\in\Z}\atop{ {\ell\ge i+1}\atop{a_{\ell-1}\ne0} } }
2^{2i}  a_{i-1}^{-2s/n} d_\ell
=\sum_{ {\ell\in\Z}\atop{a_{\ell-1}\ne0} } 
\sum_{ {i\in\Z}\atop{i\le\ell-1} } 2^{2i} a_{i-1}^{-2s/n} d_\ell
\\ &\qquad\le
\sum_{{\ell\in\Z}\atop{a_{\ell-1}\ne0}} \sum_{{i\in\Z}\atop{
i\le\ell-1}}
2^{2i} a_{\ell-1}^{-2s/n} d_\ell=
\sum_{{\ell\in\Z}\atop{a_{\ell-1}\ne0}} \sum_{k=0}^{+\infty}
2^{2(\ell-1)} 2^{-2k} a_{\ell-1}^{-2s/n} d_\ell
\\ &\qquad \le 
\sum_{{\ell\in\Z}\atop{a_{\ell-1}\ne0}}
2^{2\ell} a_{\ell-1}^{-2s/n} d_\ell =S.
\end{split}\end{equation}
Now, we fix $i\in\Z$ and~$x\in D_i$: then, for any~$j\in\Z$
with~$j\le i-2$ and any~$y\in D_j$ we have that
$$ |f(x)-f(y)|\ge 2^i-2^{j+1}\ge 2^i-2^{i-1}=2^{i-1}$$
and therefore, recalling~\eqref{X1},
\begin{eqnarray*}&& \sum_{{j\in\Z}\atop{j\le i-2}}\int_{D_j}
\frac{|f(x)-f(y)|^2}{|x-y|^{n+2s}}\,dy
\ge 2^{2(i-1)}
\sum_{{j\in\Z}\atop{j\le i-2}}\int_{D_j}
\frac{dy}{|x-y|^{n+2s}}\\ &&\qquad=
2^{2(i-1)} \int_{\CC A_{i-1}}
\frac{dy}{|x-y|^{n+2s}}
.\end{eqnarray*}
This and Lemma~\ref{5yhh}
imply that, for any~$i\in\Z$
and any~$x\in D_i$, we have that
$$ \sum_{ {j\in\Z}\atop{j\le i-2} }\int_{D_j}
\frac{|f(x)-f(y)|^2}{|x-y|^{n+2s}}\,dy
\ge c_o 2^{2i} a_{i-1}^{-2s/n},$$
for a suitable~$c_o>0$.

As a consequence,
for any~$i\in\Z$,
\begin{equation}\label{X4a}
\sum_{ {j\in\Z}\atop{j\le i-2} }\int_{D_i\times D_j}
\frac{|f(x)-f(y)|^2}{|x-y|^{n+2s}}\,d(x,y)
\ge c_o 2^{2i} a_{i-1}^{-2s/n} d_i,
\end{equation}
where~$d(x,y)$ denotes the volume element for
the product Lebesgue measure on~$\R^n\times\R^n$.
Therefore, by~\eqref{X3}, we conclude that,
for any~$i\in\Z$,
\begin{equation}\label{X4}
\sum_{ {j\in\Z}\atop{j\le i-2} }\int_{D_i\times D_j}
\frac{|f(x)-f(y)|^2}{|x-y|^{n+2s}}\,d(x,y)
\ge c_o \left[ 2^{2i} a_{i-1}^{-2s/n} a_i
-\sum_{{\ell\in\Z}\atop{\ell\ge i+1}
}2^{2i} a_{i-1}^{-2s/n} d_\ell 
\right] .\end{equation}
By \eqref{SS4} and \eqref{X4a}, we have that
\begin{equation}\label{X5bis}
\sum_{{i\in\Z}\atop{a_{i-1}\ne0}}
\sum_{{j\in\Z}\atop{j\le i-2}}\int_{D_i\times D_j}
\frac{|f(x)-f(y)|^2}{|x-y|^{n+2s}}\,d(x,y)
\ge c_o S.\end{equation}
Then, using~\eqref{X4}, \eqref{X5} and~\eqref{X5bis},
\begin{eqnarray*}
&& \sum_{{i\in\Z}\atop{a_{i-1}\ne0}}
\sum_{{j\in\Z}\atop{j\le
i-2}}\int_{D_i\times D_j}
\frac{|f(x)-f(y)|^2}{|x-y|^{n+2s}}\,d(x,y)
\\ &\ge&
c_o \left[ \sum_{{i\in\Z}\atop{a_{i-1}\ne0}}
2^{2i} a_{i-1}^{-2s/n} a_i
\,-\,
\sum_{{i\in\Z}\atop{a_{i-1}\ne0}}
\sum_{{\ell\in\Z}
\atop{\ell\ge i+1}}2^{2i} a_{i-1}^{-2s/n} d_\ell\right]
\\ &\ge& c_o
\left[ \sum_{{i\in\Z}\atop{a_{i-1}\ne0}}
2^{2i} a_{i-1}^{-2s/n} a_i
\,-\, S
\right]
\\ &\ge&  c_o
\sum_{{i\in\Z}\atop{a_{i-1}\ne0}}
2^{2i} a_{i-1}^{-2s/n} a_i
-
\sum_{{i\in\Z}\atop{a_{i-1}\ne0}}
\sum_{{j\in\Z}\atop{j\le
i-2}}\int_{D_i\times D_j}
\frac{|f(x)-f(y)|^2}{|x-y|^{n+2s}}\,d(x,y).
\end{eqnarray*}
That is, by taking the last term to
the left hand side,
\begin{equation}\label{XF}
2 \sum_{{i\in\Z}\atop{a_{i-1}\ne0}}
\sum_{{j\in\Z}\atop{j\le
i-2}}\int_{D_i\times D_j}
\frac{|f(x)-f(y)|^2}{|x-y|^{n+2s}}\,d(x,y)
\,\ge\,
c_o
\sum_{{i\in\Z}\atop{a_{i-1}\ne0}}
2^{2i} a_{i-1}^{-2s/n} a_i.
\end{equation}
On the other hand, by symmetry,
\begin{equation}\label{XFF}\begin{split}
& \int_{\R^n\times\R^n}
\frac{|f(x)-f(y)|^2}{|x-y|^{n+2s}}\,d(x,y) =
\sum_{{i,j\in\Z}} \int_{D_i\times D_j}
\frac{|f(x)-f(y)|^2}{|x-y|^{n+2s}}\,d(x,y)
\\ &\qquad =2
\sum_{{i,j\in\Z}\atop{j\le i}}
\int_{D_i\times D_j}
\frac{|f(x)-f(y)|^2}{|x-y|^{n+2s}}\,d(x,y)
\\ &\qquad\ge 2
\sum_{{i\in\Z}\atop{a_{i-1}\ne0}}
\sum_{{j\in\Z}\atop{j\le
i-2}}\int_{D_i\times D_j}
\frac{|f(x)-f(y)|^2}{|x-y|^{n+2s}}\,d(x,y).
\end{split}
\end{equation}
Then, the desired result plainly follows
from~\eqref{XF} and~\eqref{XFF}.
\end{proof}

With the above estimates, we are now in the position of
completing the
elementary proof of the Sobolev-type inequality exploited in our paper:

\begin{thm}\label{SOBOLEV}
Let $s\in(0,1)$.
Let $f:\R^n\rightarrow\R$ be measurable and compactly
supported.
Then,
\begin{equation}\label{CA4} \| f\|_{L^{2n/(n-2s)} (\R^n)}^2\le C(n,s)
\int_{\R^n} \int_{\R^n} \frac{|f(x)-f(y)|^2}{|x-y|^{n+2s}}
\,dx\,dy,\end{equation}
for a suitable constant $C(n,s)>0$.
\end{thm}

\begin{proof} Of course, we may suppose that
\begin{equation}\label{suppose}{\mbox{
the right hand side
of \eqref{CA4} is finite,}}\end{equation}
otherwise we are done.

We will prove~\eqref{CA4} under the additional assumption that
\begin{equation}\label{vega}
f\in L^\infty(\R^n).\end{equation} This 
does not affect the generality
of the result, because if~\eqref{CA4} holds for bounded functions,
then it holds for the function~$f_N$ obtained by~$f$ by
cutting at levels~$-N$ and~$+N$. Then,
denoting by~$|f|_N$ the function obtained
by cutting~$|f|$ at level~$N$, we see that~$|f|_N=|f_N|$,
so we obtain from
the Fatou Lemma that
\begin{eqnarray}\label{991}\nonumber
&& \liminf_{N\rightarrow+\infty} \| f_N\|_{L^{2n/(n-2s)}}
= \liminf_{N\rightarrow+\infty} \left( \int_{\R^n}
\Big( |f|_N\Big)^{2n/(n-2s)}
\right)^{(n-2s)/(2n)} \\ &&\qquad\ge
\left( \int_{\R^n}
|f|^{2n/(n-2s)}
\right)^{(n-2s)/(2n)} =\| f\|_{L^{2n/(n-2s)}}.
\end{eqnarray}
Also, by~\eqref{suppose} and the
Dominated Convergence Theorem, we have that
\begin{equation}\label{992}
\lim_{N\rightarrow+\infty} \int_{\R^n} \int_{\R^n} 
\frac{|f_N(x)-f_N(y)|^2}{|x-y|^{n+2s}}
\,dx\,dy
=
\int_{\R^n} \int_{\R^n} \frac{|f(x)-f(y)|^2}{|x-y|^{n+2s}}
\,dx\,dy.\end{equation}
{F}rom \eqref{991}
and~\eqref{992},
one deduces~\eqref{CA4} for~$f$ from the one 
for~$f_N$,
hence we may and do assume~\eqref{vega}.

We have, using the notation in~\eqref{t6},
\begin{eqnarray*}&&\| f\|_{L^{2n/(n-2s)} (\R^n)}^{
{2n/(n-2s)}
}=\sum_{k\in\Z}
\int_{A_k\setminus A_{k+1}} |f|^{2n/(n-2s)}(x)\,dx
\\ &&\qquad\le \sum_{k\in\Z}
\int_{A_k\setminus A_{k+1}} (2^{k+1})^{2n/(n-2s)} \,dx
\le \sum_{k\in\Z} 2^{2(k+1)n/(n-2s)} a_k
.\end{eqnarray*}
That is,
$$\| f\|_{L^{2n/(n-2s)} (\R^n)}^2
\le 4 \left( \sum_{k\in\Z} 2^{2kn/(n-2s)} a_k\right)^{(n-2s)/n}
.$$
Thus, since ${(n-2s)/n}<1$,
$$\| f\|_{L^{2n/(n-2s)} (\R^n)}^2
\le 4 \sum_{k\in\Z} 2^{2k}
a_k^{(n-2s)/n}.$$
This, \eqref{vega}, Lemma \ref{OS1} (applied with~$T:=2^2$)
and Lemma~\ref{OS2} give the claim.
\end{proof}

It may be worth to
remark that, from Lemma~\ref{5yhh}, it follows that
\begin{equation}\label{CA3bis}\int_E \int_{\CC
E}\frac{dx\,dy}{|x-y|^{n+2s}}\ge
c(n,s)\,|E|^{(n-2s)/n}\end{equation}
for all
measurable sets~$E$ with finite measure.

On the other hand, we see that~\eqref{CA4} reduces to~\eqref{CA3bis} 
when~$f=\chi_E$,
so~\eqref{CA3bis} (and thus
Lemma~\ref{5yhh}) may be seen as a Sobolev-type inequality for sets.

\vfill

\vspace{1cm}

{{\sc Ovidiu Savin}

Mathematics Department, Columbia University,

2990 Broadway, New York , NY 10027, USA.

Email: {\tt savin@math.columbia.edu}
}

\vspace{1cm}

{{\sc Enrico Valdinoci}

Dipartimento di Matematica, Universit\`a di Roma Tor Vergata,

Via della Ricerca Scientifica 1, 00133 Roma, Italy.

Email: {\tt enrico@mat.uniroma3.it}
}

\end{document}